\documentclass[sn-mathphys,Numbered]{sn-jnl}

\usepackage{commath}

\usepackage{hyperref}
\hypersetup{urlcolor=blue} 

\usepackage{tikz}
\usetikzlibrary{arrows}

\usepackage{graphicx}%
\usepackage{multirow}%
\usepackage{amsmath,amssymb,amsfonts}%
\usepackage{amsthm}%
\usepackage{mathrsfs}%
\usepackage[title]{appendix}%
\usepackage{xcolor}%
\usepackage{textcomp}%
\usepackage{manyfoot}%
\usepackage{booktabs}%
\usepackage{algorithm}%
\usepackage{algorithmicx}%
\usepackage{algpseudocode}%
\usepackage{listings}%

\theoremstyle{thmstyleone}%
\newtheorem{theorem}{Theorem}[section]
\newtheorem{lemma}{Lemma}[section]
\theoremstyle{thmstyletwo}%

\newtheorem{remark}{Remark}%

\theoremstyle{thmstylethree}%
\newtheorem{definition}{Definition}%

\begin{document}

\title[On the Rankin--Selberg $L$-function related to the Godement--Jacquet $L$-function II]{On the Rankin--Selberg $L$-function related to the Godement--Jacquet $L$-function II}

\author*[1]{\fnm{Amrinder} \sur{Kaur}}\email{amrinder1kaur@gmail.com}

\author[2]{\fnm{Ayyadurai} \sur{Sankaranarayanan}}\email{sank@uohyd.ac.in}

\affil*[1]{\orgdiv{School of Mathematics and Statistics}, \orgname{University of Hyderabad}, \orgaddress{\city{Hyderabad}, \postcode{500046}, \state{Telangana}, \country{India}}}

\affil[2]{\orgdiv{School of Mathematics and Statistics}, \orgname{University of Hyderabad}, \orgaddress{\city{Hyderabad}, \postcode{500046}, \state{Telangana}, \country{India}}}

\abstract{In this paper, we consider the $k$-th Riesz mean for the coefficients of the Rankin--Selberg $L$-function $L_{f \times f}(s)$ related to the Godement--Jacquet $L$-function with respect to $SL(n,\mathbb{Z})$. We establish an asymptotic formula for the $k$-th Riesz mean with an improved range and a better error term. As a result, we get an asymptotic relation for the partial sum of the coefficients of $L_{f \times f}(s)$.}



\maketitle

\section{Introduction} 

\noindent
For $n \geq 2$, an element $z \in \mathcal{H}^n$ takes the form $z=x \cdot y$ where 
$$ x= 
\begin{pmatrix}
1 & x_{1,2} & x_{1,3} & \cdots & x_{1,n} \\
  & 1          & x_{2,3} & \cdots & x_{2,n} \\
  &            &  \ddots  &         & \vdots \\
  &            &            & 1       & x_{n-1,n} \\
  &            &            &          & 1 \\
\end{pmatrix} 
, $$
$$ y=
\begin{pmatrix}
y_1y_2 \cdots y_{n-1} &                                  &           &      & \\
                               & y_1 y_2 \cdots y_{n-2}  &           &      & \\
                               &                                  & \ddots &      & \\
                               &                                  &          & y_1 & \\
                               &                                  &          &       & 1 \\
\end{pmatrix},
$$
with $x_{i,j} \in \mathbb{R}$ for $1 \leq i < j \leq n$ and $y_i>0$ for $1 \leq i \leq n-1$. \\

\noindent
Let $v=(v_1,v_2,\dots,v_{n-1}) \in \mathbb{C}^{n-1}$. Let $\mathfrak{D^n}$ be the center of the universal enveloping algebra of $\mathfrak{gl}(n,\mathbb{R})$ where $\mathfrak{gl}(n,\mathbb{R})$ is the Lie algebra of $GL(n,\mathbb{R}$). The function
$$ J_v(z) = \prod_{i=1}^{n-1} \prod_{j=1}^{n-1} y_i^{b_{i,j}v_j} $$
with 
$$ b_{i,j}=
\begin{cases}
ij & \text{if} \ i+j \leq n, \\
(n-i)(n-j) & \text{if} \ i+j \geq n, 
\end{cases}
$$
is an eigenfunction of every $D \in \mathfrak{D}^n$. We write 
$$ D J_v(z) = \lambda_D J_v(z) \ \text{for every} \ D \in \mathfrak{D}^n. $$ \\

\begin{definition} \cite{Dg}
Let $n \geq 2$, and let $v=(v_1,v_2,\dots,v_{n-1}) \in \mathbb{C}^{n-1}$. A Maass form for $SL(n,\mathbb{Z})$ of type $v$ is a smooth function $f \in \mathcal{L}^2(SL(n,\mathbb{Z}) \backslash \mathcal{H}^n)$ which satisfies
\begin{enumerate}
\item $f(\gamma z) = f(z)$, for all $\gamma \in SL(n,\mathbb{Z}), z \in \mathcal{H}^n$,
\item $Df(z) = \lambda_D f(z)$, for all $D \in \mathfrak{D}^n$,
\item $\int \limits_{(SL(n,\mathbb{Z}) \cap U) \backslash U} f(uz) du = 0, $\\
for all upper triangular groups $U$ of the form 
$$U = \left\{
\begin{pmatrix} 
I_{r_1} &             &           &  \\
           & I_{r_2}  &           & *\\ 
           &             & \ddots &   \\ 
           &             &           & I_{r_b}  \\ 
\end{pmatrix} \right\}, $$
with $r_1+r_2+\cdots+r_b=n$. Here $I_r$ denotes the $r \times r$ identity matrix, and $*$ denotes arbitrary real entries. \\
\end{enumerate}
\end{definition}

\noindent
A \emph{Hecke--Maass form} is a Maass form which is an eigenvector for the Hecke operators algebra. 

\noindent
Let $f(z)$ be a Hecke--Maass form of type $v=(v_1,v_2,\dots,v_{n-1}) \in \mathbb{C}^{n-1}$ for $SL(n,\mathbb{Z})$. Then it has the Fourier expansion 

\begin{align*}
f(z) &= \sum_{\gamma \in U_{n-1}(\mathbb{Z}) \backslash SL(n-1,\mathbb{Z})} \sum_{m_1=1}^{\infty} \dots \sum_{m_{n-2}=1}^{\infty} \sum_{m_{n-1} \neq 0} \frac{A(m_1,\dots,m_{n-1})}{\prod_{j=1}^{n-1} \abs{m_j}^\frac{j(n-j)}{2}} \\
& \quad \times W_J \left( M \cdot \begin{pmatrix} \gamma &  \\  & 1 \end{pmatrix} z, v, \psi_{1,\dots,1,\frac{m_{n-1}}{\abs{m_{n-1}}}}\right) ,
\end{align*}

\noindent
where 
$$ M = 
\begin{pmatrix} 
 m_1\cdots m_{n-2} \cdot \abs{m_{n-1}} &           &             &        & \\
                                                        & \ddots &              &        & \\ 
                                                        &           & m_1m_2 &        & \\ 
                                                        &           &             & m_1 & \\ 
                                                        &           &             &        & 1 
\end{pmatrix} ,
$$

$$ A(m_1,\dots,m_{n-1}) \in \mathbb{C} , \qquad  A(1,\dots,1)=1,$$

$$ \psi_{1,\dots,1,\epsilon} \left( 
\begin{pmatrix}
1 & u_{n-1} &             &            &   \\
  & 1          & u_{n-2}  &            & * \\
  &            & \ddots     & \ddots  &     \\
  &            &              & 1          & u_1 \\
  &            &              &             & 1        
\end{pmatrix}
\right) = e^{2 \pi i (u_1+\cdots +u_{n-2}+\epsilon u_{n-1})},$$

\noindent
$U_{n-1}(\mathbb{Z})$ denotes the group of $(n-1) \times (n-1)$ upper triangular matrices with $1s$ on the diagonal and an integer entry above the diagonal and $W_J$ is the Jacquet Whittaker function. \\

\begin{definition}
If $f(z)$ is a Maass form of type $(v_1,\dots,v_{n-1}) \in \mathbb{C}^{n-1}$, then
$$ \tilde{f}(z) := f(w \cdot (z^{-1})^T \cdot w), $$
$$ w = \begin{pmatrix} 
&&& (-1)^{ \left[ \frac{n}{2} \right]} \\
&& 1 & \\
& \reflectbox{$\ddots$} && \\
1&&& 
 \end{pmatrix} $$
is a Maass form of type $(v_{n-1},\dots,v_1)$ for $SL(n,\mathbb{Z})$ called the dual Maass form. If $A(m_1, \dots, m_{n-1})$ is the $(m_1, \dots, m_{n-1})$--Fourier coefficient of $f$, then $A(m_{n-1} , \dots, m_1)$ is the corresponding Fourier coefficient of $\tilde{f}$. \\
\end{definition}

\begin{definition} \cite{YjGl}
The Godement--Jacquet $L$-function $L_f(s)$ attached to $f$ is defined for $\Re(s) >1$ by 
$$ L_f(s) = \sum_{m=1}^{\infty} \frac{A(m,1,\dots,1)}{m^s} = \prod_p \prod_{i=1}^n (1-\alpha_{p,i}p^{-s})^{-1} ,$$ 
where the $\{ \alpha_{p,i} \}, 1\leq i \leq n$ are the complex roots of the monic polynomial 

$$ X^n + \sum_{r=1}^{n-1} (-1)^r A(\overbrace{1,\dots,1}^{r-1 \; \text{terms}},p,1,\dots,1) X^{n-r} +(-1)^n \in \mathbb{C}[X], \quad \text{and}$$

$$ A(\overbrace{1,\dots,1}^{r-1},p,1,\dots,1) = \sum_{1 \leq i_1 < \dots < i_r \leq n} \alpha_{p,i_1} \dots \alpha_{p,i_r}, \qquad \text{for} \; \; 1 \leq r \leq n-1  .$$ \\
\end{definition}

\begin{definition} \cite{Dg}
For $n \geq 2$, let $f,g$ be two Maass forms for $SL(n,\mathbb{Z})$ of type $v_f,v_g \in \mathbb{C}^{n-1}$, respectively, with Fourier expansions:

\begin{align*}
f(z) &= \sum_{\gamma \in U_{n-1}(\mathbb{Z}) \backslash SL(n-1,\mathbb{Z})} \sum_{m_1=1}^{\infty} \dots \sum_{m_{n-2}=1}^{\infty} \sum_{m_{n-1} \neq 0} \frac{A(m_1,\dots,m_{n-1})}{\prod_{j=1}^{n-1} \abs{m_j}^\frac{j(n-j)}{2}} \\
& \quad \times W_J \left( M \cdot \begin{pmatrix} \gamma &  \\  & 1 \end{pmatrix} z, v_f, \psi_{1,\dots,1,\frac{m_{n-1}}{\abs{m_{n-1}}}}\right) ,
\end{align*}

\begin{align*}
g(z) &= \sum_{\gamma \in U_{n-1}(\mathbb{Z}) \backslash SL(n-1,\mathbb{Z})} \sum_{m_1=1}^{\infty} \dots \sum_{m_{n-2}=1}^{\infty} \sum_{m_{n-1} \neq 0} \frac{B(m_1,\dots,m_{n-1})}{\prod_{j=1}^{n-1} \abs{m_j}^\frac{j(n-j)}{2}} \\
& \quad \times W_J \left( M \cdot \begin{pmatrix} \gamma &  \\  & 1 \end{pmatrix} z, v_g, \psi_{1,\dots,1,\frac{m_{n-1}}{\abs{m_{n-1}}}}\right). 
\end{align*}

\noindent
Let $s \in \mathbb{C}$. Then the Rankin--Selberg $L$-function, denoted as $L_{f\times g}(s)$, is defined by
$$ L_{f\times g}(s) = \zeta(ns) \sum_{m_1=1}^{\infty} \dots \sum_{m_{n-1}=1}^{\infty} \frac{A(m_1,\dots,m_{n-1}) \cdot \overline{B(m_1,\dots,m_{n-1})}}{(m_1^{n-1} m_2^{n-2} \dots m_{n-1})^s},$$ \\
which converges absolutely provided $\Re(s)$ is sufficiently large. \\
\end{definition}

\noindent
In the special case $g=f$, we have
$$ L_{f\times f}(s) = \zeta(ns) \sum_{m_1=1}^{\infty} \dots \sum_{m_{n-1}=1}^{\infty} \frac{\abs{A(m_1,\dots,m_{n-1})}^2}{(m_1^{n-1} m_2^{n-2} \dots m_{n-1})^s}$$ 
for $\Re(s) > 1$. \\

\noindent
Let $E_v(z)$ denote the minimal parabolic Eisenstein series. The $L$-function associated to $E_v$ (see \cite[Equation (10.8.5)]{Dg}) is computed as
$$ L_{E_v}(z)= \sum_{c_1=1}^{\infty} \cdots \sum_{c_{n-1}=1}^{\infty} \sum_{m=1}^{\infty} (m c_1 \cdots c_{n-1})^{-s} J_{v-\frac{1}{n}} \left( 
\begin{pmatrix}
\frac{c_1}{m} &          &                          & \\
                    & \ddots &                          & \\
                    &          & \frac{c_{n-1}}{m} & \\
                    &          &                           & 1 
\end{pmatrix} 
\right) . $$

\noindent
From \cite[Theorem 10.8.6]{Dg}, there exist functions $\lambda_i: \mathbb{C}^{n-1} \to \mathbb{C}$ satisfying $ \Re \left( \lambda_i(v) \right) =0$ if $\Re(v_i) =\frac{1}{n} (i=1,\dots,n-1)$ such that the $L$-function associated to $E_v$ is just a product of shifted Riemann zeta functions of the form 
$$ L_{E_v}(z) = \prod_{i=1}^n \zeta \left( s-\lambda_i(v) \right) .$$ \\

\noindent
We write 
$$ L_{f \times f}(s) := \sum_{m=1}^{\infty} \frac{b(m)}{m^s} \qquad \text{for} \ \Re(s) > 1.$$ 
Also, $s=\sigma+it$ and $t$ is sufficiently large. \\

In \cite{AkAs}, we proved the following two theorems. Theorem A is an unconditional result while Theorem B is a conditional result.

{\bf Theorem A.} 
\emph{
Let $n \geq 3$ be an arbitrary but fixed integer.  For $k \geq k_0(n) = \frac{n^2(n+1)}{2} +n$, we have 
$$ \sum_{m \leq x} \frac{b(m)}{k!} \left( 1-\frac{m}{x} \right)^k = \frac{Cx}{(k+1)!} + O_{n} (\log x) .$$
Here $C$ is an effective constant depending only on $f$. } \\

{\bf Theorem B.}
\emph{
Assume coefficient growth hypothesis and Lindel\"{o}f hypothesis for $L_{f \times f}(s)$. Let $n \geq 3$ be any arbitrary but fixed integer, then the asymptotic formula
$$ \sum_{m \leq x} \frac{b(m)}{k!} \left( 1-\frac{m}{x} \right)^k  = \frac{Cx}{(k+1)!} + O_{n,\epsilon}(x^{\frac{1}{2}+\epsilon})$$
holds for every positive integer $k \geq 1$. } \\

The aim of this article is twofold. First, we want to improve the range of $k$ in Theorem A with a better error term. Then, by a reduction argument, we will obtain an unconditional result, namely an asymptotic formula for the sum $\sum \limits_{m \leq x} b(m)$. Thus, we prove:

\begin{theorem} \label{t1} 
Let $n \geq 3$ be an arbitrary but fixed integer.  For $k \geq k_1(n) = \left[ \frac{n^2}{2} \right]+1 $, we have 
$$ \sum_{m \leq x} \frac{b(m)}{k!} \left( 1-\frac{m}{x} \right)^k = \frac{Cx}{(k+1)!} + O_{n} (1) .$$
Here $C$ is an effective constant depending only on $f$. \\
\end{theorem}

\begin{theorem} \label{t2}
For sufficiently large $x$, we have
$$ \sum_{m \leq x} b(m) = \frac{2^{k_1}C}{(k_1+1)} x + O_n \left( x^{1-\frac{1}{2^{k_1}}}  \right)  $$
where $k_1 = k_1(n) = \left [ \frac{n^2}{2} \right] +1. $ \\
\end{theorem}

\begin{remark}
\normalfont{
When proving Theorem B in \cite{AkAs}, we assumed two hypotheses: one was the coefficient growth hypothesis and another was the Lindel\"{o}f hypothesis for $L_{f \times f}(s)$. From Lemma \ref{l1}, we can see that the coefficient growth hypothesis is redundant. Just with the assumption of Lindel\"{o}f hypothesis for $k=1$, we have
$$ \sum_{m \leq x} b(m) \left( 1- \frac{m}{x} \right) = \frac{Cx}{2} + O(x^{\frac{1}{2}+\epsilon}) .$$
Using Lemma \ref{l5}, we observe that conditionally we get
$$ \sum_{m \leq x} b(m) = Cx + O(x^{\frac{3}{4}+\epsilon}) .$$
Though the error term obtained in Theorem \ref{t2} is weaker than what is expected, it is an unconditional result. \\
}
\end{remark}

\begin{remark}
\normalfont{
We note that the reduction process in Lemma \ref{l5} originated in \cite{Ae} due to Ingham. This idea has been successfully exploited under various circumstances by several researchers. For instance, see \cite{RbKr, RbOrPs, RbPs}. \\
}
\end{remark}

\emph{Throughout the paper, we assume that $f$ is a self-dual Hecke--Maass form for $SL(n,\mathbb{Z})$ and $\epsilon$ is any small positive constant.}

\section{Preliminaries} 

\noindent
In this section, we present some necessary properties of the Rankin--Selberg $L$-function which are used later.

\subsection{Euler Product} 

\noindent 
Fix $n \geq 2$. Let $f,g$ be two Maass forms for $SL(n,\mathbb{Z})$ with Euler products

$$ L_f(s) = \sum_{m=1}^{\infty} \frac{A(m,1,\dots,1)}{m^s} = \prod_p \prod_{i=1}^n (1-\alpha_{p,i}p^{-s})^{-1} ,$$

$$ L_g(s) = \sum_{m=1}^{\infty} \frac{B(m,1,\dots,1)}{m^s} = \prod_p \prod_{i=1}^n (1-\beta_{p,i}p^{-s})^{-1} ,$$

\noindent
then $L_{f \times g}(s)$ will have an Euler product of the form:

$$ L_{f \times g}(s) = \prod_p \prod_{i=1}^n \prod_{j=1}^n (1-\alpha_{p,i} \overline{\beta_{p,j}} p^{-s})^{-1} .$$ \\

\subsection{Functional Equation} 

\noindent
For $n \geq 2$, let $f,g$ be two Maass forms of types $v_f,v_g$ for $SL(n,\mathbb{Z})$  whose associated $L$-functions $L_f,L_g$ satisfy the functional equations:
\begin{align*}
\Lambda_f(s) &:= \prod_{i=1}^n \pi^{\frac{-s+\lambda_i(v_f)}{2}} \Gamma \left( \frac{s-\lambda_i(v_f)}{2} \right) L_f(s) \\
&= \Lambda_{\tilde{f}}(1-s), \\
\Lambda_g(s) &:= \prod_{j=1}^n \pi^{\frac{-s+\lambda_j(v_g)}{2}} \Gamma \left( \frac{s-\lambda_j(v_g)}{2} \right) L_g(s) \\
&= \Lambda_{\tilde{g}}(1-s),
\end{align*}
where $\tilde{f},\tilde{g}$ are the Dual Maass forms. \\

\noindent
Then the Rankin--Selberg $L$-function $L_{f \times g}(s)$ has a meromorphic continuation to all $s \in \mathbb{C}$ with at most a simple pole at $s=1$ with residue proportional to $\langle f,g \rangle$, the Petersson inner product of $f$ with $g$.
$L_{f \times g}(s)$ satisfies the functional equation:
\begin{align*}
\Lambda_{f \times g}(s) &:= \prod_{i=1}^n \prod_{j=1}^n \pi^{\frac{-s+\lambda_i(v_f)+\overline{\lambda_j(v_g)}}{2}} \Gamma \left( \frac{s-\lambda_i(v_f)-\overline{\lambda_j(v_g)}}{2} \right) L_{f \times g}(s) \\
&= \Lambda_{\tilde{f} \times \tilde{g}}(1-s).
\end{align*}

\noindent
From Equation (10.8.5) and Remark 10.8.7 of \cite{Dg}, the powers of $\pi$ take the much simpler form:
$$ \prod_{i=1}^n \pi^{\frac{-s+\lambda_i(v)}{2}} = \pi^{\frac{-ns}{2}} , \qquad \prod_{i=1}^n \prod_{j=1}^n \pi^{\frac{-s+\lambda_i(v_f)+\overline{\lambda_j(v_g)}}{2}} = \pi^{\frac{-n^2 s}{2}} .$$

\noindent
Hence, we get
\begin{align*}
\Lambda_{f \times g}(s) &:= \pi^{\frac{-n^2 s}{2}} \prod_{i=1}^n \prod_{j=1}^n \Gamma \left( \frac{s-\lambda_i(v_f)-\overline{\lambda_j(v_g)}}{2} \right) L_{f \times g}(s) \\
&= \Lambda_{\tilde{f} \times \tilde{g}}(1-s).
\end{align*}

\noindent
We take $g=f$ and $f$ to be a self-dual Maass form of type $v$ so that
\begin{align*}
\Lambda_{f \times f}(s) &:= \pi^{\frac{-n^2 s}{2}} \prod_{i=1}^n \prod_{j=1}^n \Gamma \left( \frac{s-\lambda_i(v)-\overline{\lambda_j(v)}}{2} \right) L_{f \times f}(s) \\
&= \Lambda_{f \times f}(1-s). \\
\end{align*}

\subsection{Bound for the conversion factor} 

Let $f$ be a self-dual Hecke--Maass form. Then we have the functional equation 
$$ \Lambda_{f \times f}(s) = \Lambda_{f \times f}(1-s) .$$

If we write $L_{f \times f}(s) = \chi_{f \times f}(s) L_{f \times f}(1-s)$, then from our work in \cite{AkAs} the conversion factor $\chi_{f \times f}(s)$ can be written as

$$ \chi_{f \times f}(\sigma+it) \ll \abs{t}^{n^2 \left( \frac{1}{2}-\sigma \right)}.$$

This bound is true in any fixed vertical strip $a \leq \sigma \leq b$ and sufficiently large $t$. \\
\emph{Hereafter, throughout the paper, we assume $n \geq 3$.} \\

\section{Some Lemmas} 

\begin{lemma} \label{l1}
For $\Re(s) \geq 1+\epsilon$, $L_{f \times f}(s)$ is absolutely convergent.
\end{lemma}

\begin{proof}
The Rankin--Selberg $L$-function $L_{f \times f}(s)$ has a meromorphic continuation to all $s \in \mathbb{C}$ with a simple pole at $s=1$. It is easy to see that
$$ L_{f\times f}(s) = \zeta(ns) \sum_{m_1=1}^{\infty} \dots \sum_{m_{n-1}=1}^{\infty} \frac{\abs{A(m_1,\dots,m_{n-1})}^2}{(m_1^{n-1} m_2^{n-2} \dots m_{n-1})^s}$$ 
implies that the coefficients $b(m)$ are non-negative. Landau's lemma asserts that a Dirichlet series with non-negative coefficients must be absolutely convergent up to its first pole. Hence, $L_{f \times f}(s)$ is absolutely convergent in the half-plane $\Re(s) \geq 1+\epsilon$.

\end{proof}

\begin{lemma} \label{l2}
For sufficiently large $t$, we have
$$ L_{f \times f}(s) \ll \left(\, \abs{t}+10 \right)^{\frac{n^2}{2}(1+\epsilon-\sigma)}  $$
uniformly for $-\epsilon \leq \sigma \leq 1+\epsilon$.
\end{lemma}

\begin{proof}
We prove along the same lines as in \cite[Lemma 3.5]{As}. From Lemma \ref{l1}, we have
$$ \abs{L_{f \times f}(1+\epsilon+it)} \ll 1 , $$
and by the functional equation
\begin{align*}
\abs{L_{f \times f}(-\epsilon+it)} &= \abs{\chi_{f \times f}(-\epsilon+it) L_{f \times f}(1+\epsilon-it)} \\
&\ll \left(\, \abs{t}+10 \right)^{n^2 \left( \frac{1}{2} + \epsilon \right)}. \\
\end{align*}

\noindent
Now we apply the maximum modulus principle to the function 
$$ F(w) = L_{f \times f}(w) e^{(w-s)^2} X^{w-s} $$
in the rectangle 

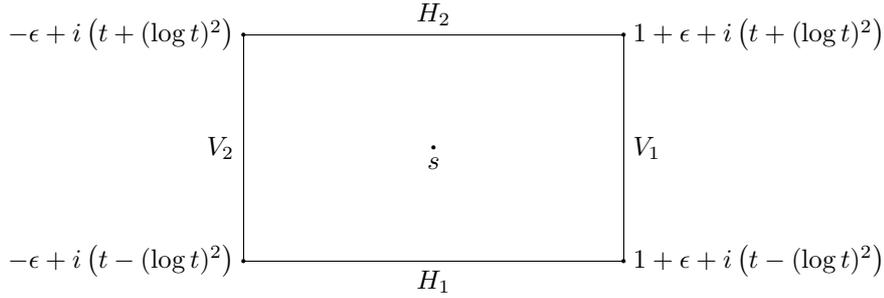
\begin{figure}[h!]
\begin{center}
\begin{tikzpicture}

\draw (0,0) -> (5,0) node[midway,below]{$H_1$};
\draw (0,0) -> (0,3) node[midway,left]{$V_2$};
\draw (0,3) -> (5,3) node[midway,above]{$H_2$};
\draw (5,0) -> (5,3) node[midway,right]{$V_1$};

\node at (0,0) {$\boldsymbol{\cdot}$};
\node at (0,3) {$\boldsymbol{\cdot}$};
\node at (5,0) {$\boldsymbol{\cdot}$};
\node at (5,3) {$\boldsymbol{\cdot}$};
\node at (2.5,1.5) {$\boldsymbol{\cdot}$};

\coordinate [label= left:$-\epsilon+i \left( t-(\log t)^2 \right)$] (A) at (0,0);
\coordinate [label= left:$-\epsilon+i \left( t+(\log t)^2 \right)$] (D) at (0,3);
\coordinate [label= right:$1+\epsilon+i \left( t-(\log t)^2 \right)$] (B) at (5,0);
\coordinate [label= right:$1+\epsilon+i \left( t+(\log t)^2 \right)$] (C) at (5,3);
\coordinate [label= below:$s$] (E) at (2.5,1.5);

\end{tikzpicture}
\end{center}
\end{figure}

so that
$$ \abs{L_{f \times f}(s)} \ll V_1+V_2+H_1+H_2 .$$
Here $V_1,\,V_2$ are the contributions from the vertical lines and $H_1,\,H_2$ are the contributions from the horizontal lines. \\

Let $w=u+iv$ and $s=\sigma+it$. As
\begin{align*}
\exp \{ (w-s)^2 \} &= \exp \{ (u-\sigma)^2 - (v-t)^2 +2i(u-\sigma)(v-t) \} \\
\abs{ \exp \{ (w-s)^2 \} } &= \exp \{ (u-\sigma)^2 - (v-t)^2 \} \\
&\ll \exp \{ -(\log t)^2 \},
\end{align*}

we see that $\exp \{ (w-s)^2 \}$ decays exponentially for large $t$ on horizontal lines. Thus,
\begin{align*}
H_1 &\ll 1, \; H_2 \ll 1, \; V_1 \ll X^{1+\epsilon-\sigma} \\
V_2 &\ll \left(\, \abs{t}+10 \right)^{n^2 \left( \frac{1}{2} + \epsilon \right)} X^{-\epsilon-\sigma}.
\end{align*}

\noindent
Therefore,
$$ \abs{L_{f \times f}(s)} \ll X^{1+\epsilon-\sigma} + \left(\, \abs{t}+10 \right)^{n^2 \left( \frac{1}{2} + \epsilon \right)} X^{-\epsilon-\sigma} +1 .$$ \\

\noindent
We choose $X$ such that
\begin{align*}
X^{1+\epsilon-\sigma} &\sim \left(\, \abs{t}+10 \right)^{n^2 \left( \frac{1}{2} + \epsilon \right)} X^{-\epsilon-\sigma} \\
\text{i.e.,} \; X &\sim \left(\, \abs{t}+10 \right)^{\frac{n^2}{2}} 
\end{align*}
so that 
\begin{align*}
\abs{L_{f \times f}(s)} &\ll \left(\, \abs{t}+10 \right)^{\frac{n^2}{2}(1+\epsilon-\sigma)} .
\end{align*}
This completes the proof of this lemma. \\
\end{proof}

\begin{lemma} \label{l3}
For $0 \leq \Re(s) \leq 1+\epsilon$, we have uniformly 
$$ L_{f \times f}(s) \ll \left(\, \abs{t}+10 \right)^{\frac{n^2}{2} + \epsilon}.$$
\end{lemma}

\begin{proof}
Follows from Lemma \ref{l2}. \\
\end{proof}

\begin{lemma} \label{l4}
Let $c$ and $y$ be any positive real numbers and $T$ is sufficiently large. Then we have,
\begin{equation*}
\frac{1}{2 \pi i} \int_{c-iT}^{c+iT} \frac{y^s}{s(s+1) \dots (s+k)} ds = 
\begin{cases}
\frac{1}{k!} \left( 1-\frac{1}{y} \right)^k + O \left( \frac{4^ky^c}{T^k} \right) &, y \geq 1 \\
O \left( \frac{1}{T^k} \right) &, 0<y \leq 1.
\end{cases}
\end{equation*}
\end{lemma} 

\begin{proof}
See \cite[Lemma 3.2]{AsSk}. \\
\end{proof}

\begin{remark}
\normalfont{
Let 
$$ B(x) = \frac{1}{x} \int_1^x A(t) \ dt .$$
If we know the asymptotic formula for $A(x)$, we can find the asymptotic relation for $B(x)$. But the converse is not true. However, if $A(x)$ is monotonic, then using the asymptotic formula for $B(x)$, we can deduce the asymptotic relation for $A(x)$. \\
}
\end{remark}

\begin{lemma} \label{l5}
Let $A(x)$ be a monotonically increasing function such that
\begin{flalign*}
& & B(x) &= \frac{1}{x} \int_1^x A(t) \ dt .  &\\ 
&\text{If} & B(x) &= cx+ O \left( \frac{x}{E(x)} \right), &\\ 
&\text{then} & A(x) &= 2cx+ O \left( \frac{x}{\sqrt{E(x)}} \right). \\
\end{flalign*}
\end{lemma}

\begin{proof}
Since
$$ B(x) = \frac{1}{x} \int_1^x A(t) \ dt ,$$
we have
$$ (x+\delta)B(x+\delta) -xB(x) = \int_x^{x+\delta} A(t) \ dt >A(x) \delta $$
where $\delta=o(x)$ is chosen later. 
Thus
\begin{align*}
A(x) &< \left( 1+\frac{x}{\delta} \right) \left( cx +c\delta + O \left( \frac{x}{E(x)} \right) \right) -\frac{x}{\delta} \left( cx + O \left( \frac{x}{E(x)} \right) \right)\\
&= cx + c\delta + O \left( \frac{x}{E(x)} \right) + \frac{cx^2}{\delta} +cx + O \left( \frac{x^2}{\delta E(x)} \right) -\frac{cx^2}{\delta} + O \left( \frac{x^2}{\delta E(x)} \right) \\
&= 2cx + c\delta + O \left( \frac{x^2}{\delta E(x)} \right).
\end{align*}

The parameter $\delta$ is chosen such that 
$$ \frac{x^2}{\delta E(x)} < \delta $$
i.e.,
$$ \delta > \frac{x}{\sqrt{E(x)}}. $$

Thus, we get
$$ A(x) < 2cx + \left( \frac{x}{\sqrt{E(x)}} \right). $$ \\

Also,
$$ xB(x) - (x-\delta)B(x-\delta) = \int_{x-\delta}^x A(t) \ dt < A(x) \delta $$
gives
\begin{align*}
A(x) &> \frac{x}{\delta} \left( cx + O \left( \frac{x}{E(x)} \right) \right) + \left( 1- \frac{x}{\delta} \right) \left( cx -c\delta + O \left( \frac{x}{E(x)} \right) \right)  \\
&= \frac{cx^2}{\delta} + O \left( \frac{x^2}{\delta E(x)} \right) +cx -c\delta + O \left( \frac{x}{E(x)} \right) -\frac{cx^2}{\delta} +cx + O \left( \frac{x^2}{\delta E(x)} \right) 
\end{align*}

We choose $\delta$ so that 
$$ \frac{x^2}{\delta E(x)} < \delta $$
i.e.,
$$ \delta > \frac{x}{\sqrt{E(x)}}. $$

Thus, we get
$$ A(x) = 2cx + \left( \frac{x}{\sqrt{E(x)}} \right). $$ \\

\end{proof}

\section{Proof of Theorem \ref{t1}}
\noindent
Let $y =\frac{x}{m} \geq 1$ and $c=1+\epsilon$ in Lemma \ref{l4} so that

\begin{align*}
\frac{1}{k!} \left( 1-\frac{m}{x} \right)^k &= \frac{1}{2 \pi i} \int_{1+\epsilon-iT}^{1+\epsilon+iT} \frac{ \left( \frac{x}{m} \right)^s}{s(s+1)\dots(s+k)} ds + O \left( \frac{4^k x^{1+\epsilon}}{T^k m^{1+\epsilon}} \right). 
\end{align*}
Hence,
\begin{align*}
\sum_{m \leq x} \frac{b(m)}{k!} \left( 1-\frac{m}{x} \right)^k &= \sum_{m \leq x} \frac{b(m)}{2 \pi i} \int_{1+\epsilon-iT}^{1+\epsilon+iT} \frac{ \left( \frac{x}{m} \right)^s}{s(s+1)\dots(s+k)} ds \\
&\quad+ O \left( \frac{4^k x^{1+\epsilon}}{T^k} \sum_{m \leq x} \frac{b(m)}{m^{1+\epsilon}} \right) \\
&= \frac{1}{2 \pi i} \int_{1+\epsilon-iT}^{1+\epsilon+iT} \frac{L_{f \times f}(s) x^s}{s(s+1)\dots(s+k)} ds + O \left( \frac{4^k x^{1+\epsilon}}{T^k} \right).
\end{align*}
Summation and integral can be interchanged because of absolute convergence. Now we move the line of integration to $\Re(s)=0$. 

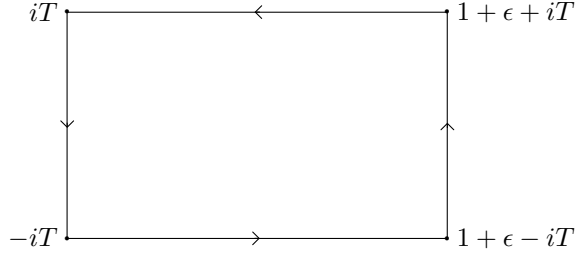
\begin{figure}[h!]
\begin{center}
\begin{tikzpicture}

\draw (0,3) -- (0,0) node[pos=0.5]{\tikz \draw[- angle 90] (0,1 pt) -- (0,-1 pt);};
\draw (0,0) -- (5,0) node[pos=0.5]{\tikz \draw[- angle 90] (-1 pt,0) -- (1 pt,0);};
\draw (5,0) -- (5,3) node[pos=0.5]{\tikz \draw[- angle 90] (0,-1 pt) -- (0,1 pt);};
\draw (5,3) -- (0,3) node[pos=0.5]{\tikz \draw[- angle 90] (1 pt,0) -- (-1 pt,0);};

\node at (0,0) {$\boldsymbol{\cdot}$};
\node at (0,3) {$\boldsymbol{\cdot}$};
\node at (5,0) {$\boldsymbol{\cdot}$};
\node at (5,3) {$\boldsymbol{\cdot}$};

\coordinate [label= left:$-iT$] (A) at (0,0);
\coordinate [label= left:$iT$] (D) at (0,3);
\coordinate [label= right:$1+\epsilon-iT$] (B) at (5,0);
\coordinate [label= right:$1+\epsilon+iT$] (C) at (5,3);

\end{tikzpicture}
\end{center}
\end{figure}

\noindent
By Cauchy's residue theorem,
\begin{align*}
&\frac{1}{2 \pi i} \left[ \int_{1+\epsilon-iT}^{1+\epsilon+iT} + \int_{1+\epsilon+iT}^{iT} + \int_{iT}^{-iT} + \int_{-iT}^{1+\epsilon-iT} \right] \frac{L_{f \times f}(s)x^s}{s(s+1)\dots (s+k)} ds \\
&= \text{Res}_{s=1} \frac{L_{f \times f}(s) x^s}{s(s+1)\dots (s+k)} \\
&= \lim_{s \to 1} \frac{(s-1)L_{f \times f}(s)x^s}{s(s+1)\dots (s+k)} \\
&= \frac{Cx}{(k+1)!}
\end{align*}
where $C=\lim \limits_{s \to 1}  (s-1) L_{f \times f}(s)$, depends on $f$. \\

\noindent
Hence,
\begin{align*}
&\frac{1}{2 \pi i} \int_{1+\epsilon-iT}^{1+\epsilon+iT} \frac{L_{f \times f}(s)x^s}{s(s+1)\dots (s+k)} ds \\
&= \frac{Cx}{(k+1)!} + \frac{1}{2 \pi i} \left[ \int_{iT}^{1+\epsilon+iT} + \int_{-iT}^{iT} + \int_{1+\epsilon-iT}^{-iT} \right] \frac{L_{f \times f}(s)x^s}{s(s+1)\dots (s+k)} ds . \\
\end{align*}

\noindent
Horizontal line contributions are in absolute value:
\begin{align*}
&\abs{ \frac{1}{2 \pi i} \int_{iT}^{1+\epsilon+iT} \frac{L_{f \times f}(s)x^s}{s(s+1)\dots (s+k)} ds } \\
&\quad= \abs{ \frac{1}{2 \pi i} \int_0^{1+\epsilon} \frac{L_{f \times f}(\sigma+iT)x^{\sigma+iT}}{(\sigma+iT)(\sigma+iT+1)\dots (\sigma+iT+k)} d\sigma } \\
&\quad \leq \frac{1}{2 \pi} \int_0^{1+\epsilon} \frac{\abs{L_{f \times f}(\sigma+iT)}x^{\sigma}}{T^{k+1}} d\sigma \\
&\quad \ll T^{\frac{n^2}{2}-k-1+\epsilon}x^{1+\epsilon} . \\
\end{align*}

\noindent
The left vertical line contribution is:
\begin{align*}
\frac{1}{2 \pi i} \int_{-iT}^{iT} \frac{L_{f \times f}(s)x^s}{s(s+1)\dots (s+k)} ds &=  \frac{1}{2 \pi i} \int_{\abs{t} \leq t_0, \atop \sigma=0} \frac{L_{f \times f}(s)x^s}{s(s+1)\dots (s+k)} ds \\
& \quad + \frac{1}{2 \pi i} \int_{t_0 \leq \abs{t} \leq T, \atop \sigma = 0} \frac{L_{f \times f}(s)x^s}{s(s+1)\dots (s+k)} ds.\\
\end{align*}

\noindent
We note that
\begin{align*}
&\abs{ \frac{1}{2 \pi i} \int_{\abs{t} \leq t_0, \atop \sigma =0} \frac{L_{f \times f}(s)x^s}{s(s+1)\dots (s+k)} ds } \\
&= \abs{ \frac{1}{2 \pi i} \int_{\abs{t} \leq t_0} \frac{L_{f \times f} (it) x^{it}}{ (it) (it+1) \dots ( it+k)} i dt } \\
& \leq \frac{1}{2 \pi} \int_{\abs{t} \leq t_0} \frac{t^{\frac{n^2}{2}-1+\epsilon}}{k!} dt \\
& \ll_n 1
\end{align*}
and
\begin{align*}
&\abs{ \frac{1}{2 \pi i} \int_{t_0 \leq \, \abs{t} \leq T, \atop \sigma = 0} \frac{L_{f \times f}(s)x^s}{s(s+1)\dots (s+k)} ds } \\
&= \abs{ \frac{1}{2 \pi i} \int_{t_0 \leq \, \abs{t} \leq T} \frac{L_{f \times f} (it) x^{it}}{(it) (it+1) \dots (it+k)} i dt } \\
& \leq \frac{1}{2 \pi} \int_{t_0 \leq \, \abs{t} \leq T} \frac{t^{\frac{n^2}{2}+\epsilon}}{t^{k+1}} dt \\
& \ll T^{\frac{n^2}{2} - k +\epsilon}. \\
\end{align*}

\noindent
Hence,
\begin{align*}
\frac{1}{2 \pi i} \int_{1+\epsilon-iT}^{1+\epsilon+iT} \frac{L_{f \times f}(s)x^s}{s(s+1)\dots(s+k)} ds &= \frac{Cx}{(k+1)!} + O(T^{\frac{n^2}{2}-k-1+\epsilon}x^{1+\epsilon}) \\
&\quad + O (T^{\frac{n^2}{2}-k+\epsilon}) +O_n(1). 
\end{align*}
This implies that
\begin{align*}
\sum_{m \leq x} \frac{b(m)}{k!} \left( 1- \frac{m}{x} \right)^k &= \frac{Cx}{(k+1)!} + O(T^{\frac{n^2}{2}-k-1+\epsilon}x^{1+\epsilon}) + O (T^{\frac{n^2}{2}-k+\epsilon}) \\
&\quad + O ( T^{-k} x^{1+\epsilon} ) +O_n(1). \\
\end{align*}

\noindent
First we choose $T=\frac{x}{10}$ so that
\begin{align*}
\sum_{m \leq x} \frac{b(m)}{k!} \left( 1-\frac{m}{x} \right)^k = \frac{Cx}{(k+1)!} + O(x^{\frac{n^2}{2}-k+\epsilon})  + O(x^{\frac{n^2}{2}-k+\epsilon}) +O(x^{1-k+\epsilon}) + O_n(1) .\\
\end{align*}

\noindent
Thus for $k \geq k_1(n) = \left[ \frac{n^2}{2} \right]+1$, we finally arrive at
$$ \sum_{m \leq x} \frac{b(m)}{k!} \left( 1-\frac{m}{x} \right)^k = \frac{Cx}{(k+1)!} + O_{n} (1) $$
which holds good for all integers $k \geq k_1(n)$. \\

\section{Proof of Theorem \ref{t2}}

From Theorem \ref{t1} with $k=k_1$ we have,
$$ \sum_{m \leq x} \frac{b(m)}{k_1!} \left( 1-\frac{m}{x} \right)^{k_1} = \frac{Cx}{(k_1+1)!} + O_{n} (1) . $$

Note that
\begin{align*}
\sum_{m \leq x} \frac{b(m)}{k_1!} \left( 1-\frac{m}{x} \right)^{k_1} &= \sum_{m \leq x} \frac{b(m)}{k_1!} \left( 1-\frac{m}{x} \right)^{k_1-1} \left( 1-\frac{m}{x} \right) \\
&= \frac{1}{x} \sum_{m \leq x} \frac{b(m)}{k_1!} \left( 1-\frac{m}{x} \right)^{k_1-1} (x-m) \\
&= \frac{1}{x} \sum_{m \leq x} \frac{b(m)}{k_1!} \left( 1-\frac{m}{x} \right)^{k_1-1} \int_m^x dt \\
&= \frac{1}{x} \int_1^x \left( \sum_{m \leq t} \frac{b(m)}{k_1!} \left( 1-\frac{m}{t} \right)^{k_1-1} \right) dt. \\
\end{align*}

Using Lemma \ref{l5} with $E(x) = 10x$, we can find the $(k_1-1)$-th Riesz mean. In particular, we get
$$ \sum_{m \leq x} \frac{b(m)}{k_1!} \left( 1-\frac{m}{x} \right)^{k_1-1} = \frac{2Cx}{(k_1+1)!} + O_{n} (x^{1-\frac{1}{2}} ) . $$

Once again using Lemma \ref{l5}, we get
$$ \sum_{m \leq x} \frac{b(m)}{k_1!} \left( 1-\frac{m}{x} \right)^{k_1-2} = \frac{2^2Cx}{(k_1+1)!} + O_{n} (x^{1-\frac{1}{2^2}}  ) . $$

Repeatedly using the result in Lemma \ref{l5} $k_1$ times, we get
$$ \sum_{m \leq x} \frac{b(m)}{k_1!} = \frac{2^{k_1}Cx}{(k_1+1)!} + O_n \left( x^{1-\frac{1}{2^{k_1}}}  \right) . $$
This proves the theorem. \\

\backmatter

\bmhead{Acknowledgments}
The authors are thankful to Prof. R. Balasubramanian for some fruitful discussions related to this paper. The first author is thankful to UGC for its supporting NET Senior Research Fellowship with UGC Ref. No. : 1004/(CSIR--UGC NET Dec. 2017). \\

\section*{Declarations}

\begin{itemize}
\item Funding - The first author is supported by University Grants Commission’s NET Senior Research Fellowship (Ref. No. 1004/(CSIR–UGC NET Dec. 2017)).
\item Conflict of interest/Competing interests - The authors have no conflicts of interest to declare.
\item Ethics approval - Not Applicable
\item Consent to participate - Not Applicable
\item Consent for publication - The authors give their consent for the publication of this article.
\item Availability of data and materials - Not Applicable
\item Code availability - Not Applicable
\item Authors' contributions - The authors have equally contributed to this work. \\
\end{itemize}


\end{document}